\documentclass[a4paper,11 pt,reqno]{amsart}
\usepackage{a4,amsbsy,amscd,amssymb,amstext,amsmath,latexsym,oldgerm,enumerate,verbatim,graphicx,xy}

\makeatletter                                       
\def\l@subsection{\@tocline{2}{0pt}{2.5pc}{2.5pc}{}}
\makeatother                                        

\makeatletter                                                  
\def\chapter{\clearpage\thispagestyle{plain}\global\@topnum\z@ 
\@afterindenttrue \secdef\@chapter\@schapter}
\makeatother

\newtheorem{thm} {Theorem} [section]
\newtheorem{prop}{Proposition} [section]
\newtheorem{lem} {Lemma} [section]

\newtheorem{cornn}{Corollary}

\theoremstyle{definition}

\newtheorem{rem} {Remark} [section]
\newtheorem{rems} [rem]{Remarks}

\newcommand{\mf}{\mathfrak}
\newcommand{\mc}{\mathcal}
\newcommand{\mb}{\mathbb}

\newcommand{\ov}{\overline}

\newcommand{\sm}{\setminus}         

\newcommand{\ot}{\otimes}           
\newcommand{\la}{\langle}
\newcommand{\ra}{\rangle}

\newcommand{\Hom}{{\rm Hom}}        
\newcommand{\Mor}{{\rm Mor}}

\newcommand{\Vect}{{\rm Vect}}

\newcommand{\ind}{{\rm ind}}

\newcommand{\Spec}{{\rm Spec}}

\newcommand{\Ker}{{\rm Ker}}

\newcommand{\rk}{{\rm rk}}

\newcommand{\id}{{\rm id}}

\let\ttie\t
\newcommand{\tie}[1]{{\let\t\ttie \ttie#1}}
\renewcommand{\t}{\mf{t}}  

\newcommand{\ad}{{\rm ad}}               

\newcommand{\Lie}{{\rm Lie}}
\newcommand{\Dist}{{\rm Dist}}
\newcommand{\GL}{{\rm GL}}
\newcommand{\PGL}{{\rm PGL}}
\newcommand{\SL}{{\rm SL}}

\xyoption{all}

\DeclareMathAlphabet\mbc{OMS}{cmsy}{b}{n}

\begin{document}

\title{Embeddings of spherical homogeneous spaces in characteristic $p$}

\begin{abstract}
Let $G$ be a reductive group over an algebraically closed field of characteristic $p>0$.
We study properties of embeddings of spherical homogeneous $G$-spaces. We look at Frobenius splittings,
canonical or by a $(p-1)$-th power, compatible with certain subvarieties.
We show the existence of rational $G$-equivariant resolutions by toroidal embeddings, and give results about
cohomology vanishing and surjectivity of restriction maps of global sections of line bundles.
We show that the class of homogeneous spaces for which our results hold contains the symmetric homogeneous spaces
in characteristic $\ne 2$ and is closed under parabolic induction.
\end{abstract}

\author[R.\ H.\ Tange]{Rudolf Tange}


\maketitle{}
\markright{\MakeUppercase{Embeddings of spherical homogeneous spaces in characteristic} $p$}

\section*{Introduction}
Let $k$ be an algebraically closed field of prime characteristic $p$ and let $G$ be a connected reductive group over $k$.
Let $H$ be a closed $k$-subgroup scheme of $G$ such that the homogeneous space $G/H$ (see \cite[III.3.5.4]{DG}) is spherical.
We are interested in properties of embeddings (always normal and equivariant) of $G/H$: existence of equivariant ``rational" resolutions
by toroidal embeddings, (canonical) Frobenius splittings, compatible splittings, cohomology vanishing, surjectivity of restriction maps of global sections of line bundles.
The idea is to show that there is a class of spherical homogeneous spaces whose embeddings have nice properties, which contains symmetric homogeneous spaces
in characteristic $\ne 2$ and the reductive group embeddings in any characteristic, and is closed under parabolic induction.
Previously, special cases were studied in \cite{Str} (the wonderful compactification of an adjoint group), \cite{BriI} ($p$ large), \cite{DeC-Spr} (the wonderful compactification of
an ``adjoint" symmetric homogeneous space), \cite{Rit} (arbitrary embeddings of reductive groups), \cite{T} (embeddings of homogeneous spaces induced from reductive groups).

We describe some of the issues that arise in this general approach. Let $x\in G/H$ be the image of the neutral element $e$ under the canonical map $G\to G/H$,
let $B$ be a Borel subgroup of $G$ such that the orbit $B\cdot x$ is open in $G/H$ and let $T$ be a maximal torus of $B$.
In \cite{T} it was shown that, after reducing to the toroidal case, one has a $B$-canonical splitting by a $(p-1)$-th power which compatibly splits the $B$-stable prime divisors,
and this implies certain cohomology vanishing results.
In more general situations, e.g. symmetric homogeneous spaces, this need no longer be true.
Firstly, one cannot construct a $B$-canonical splitting for the wonderful compactification which is a $(p-1)$-th power, since the closed orbit is $G/P$ for some $P$ which is not necessarily a Borel subgroup,
and therefore its $B$-canonical splitting is in general not a $(p-1)$-th power.
Since it seems necessary to have a splitting by a $(p-1)$-th power to get cohomology vanishing results using the Mehta-Van der Kallen theorem \cite[Thm.~1.2.12, 1.3.14]{BriKu}, and $B$-canonical
splitting are of interest for results about good filtrations we use both types of splittings. 
Furthermore, one cannot expect a splitting which compatibly splits all $B$-stable prime divisors, since their schematic intersections need not be reduced, see
\cite[p266]{Bri2}.
Finally, one cannot expect the ``usual" cohomology vanishing results for arbitrary spherical varieties in characteristic $p$, see \cite[Example~3]{HL}
for an example of an ample line bundle on a projective spherical variety with nonzero $H^1$. So we need to make some assumptions on the homogeneous space.

The paper is organised as follows. In Section~\ref{s.prelim} we introduce the notation, state the basic assumptions (A1)-(A4), and prove some technical lemmas.
In Section~\ref{s.frobsplit_ratres} we prove our three main results, each result requires one of two extra assumptions (B1) and (B2).
As it turns out, (B2) implies (B1), see Remark~\ref{rem.cansplit}.2.
In Proposition~\ref{prop.toroidal} we show that, assuming (B1), every smooth toroidal embedding is split by a $(p-1)$-th power
which compatibly splits the $G$-stable closed subvarieties. Theorem~\ref{thm.resolution} states that, assuming (B1), every $G/H$-embedding has an equivariant rational resolution by
a toroidal embedding. This then leads to results about Frobenius splitting, cohomology vanishing, and surjectivity of restriction maps of global sections of line bundles.
The proofs of Proposition~\ref{prop.toroidal} and Theorem~\ref{thm.resolution} use ideas from \cite{Str}, \cite{Rit}, \cite{BriKu} together with the fact that $G/P$
is split by a $(p-1)$-th power, Lemma~\ref{lem.GmodP}, and the fact that we have a certain expression for the canonical divisor which contains every $B$-stable prime divisor with strictly negative coefficient, Lemma~\ref{lem.pull-back}.
Theorem~\ref{thm.cansplit} states that, assuming (B2), every $G/H$-embedding is $B$-canonically split compatible with all $G$-stable closed subvarieties.
In Section~\ref{s.symmetric} we show that the class of homogeneous spaces
which satisfy (A1)-(A4) and (B2) contains the symmetric spaces.
In Section~\ref{s.induction} we show that this class is stable under parabolic induction.

\section{Preliminaries}\label{s.prelim}
We retain the notation $k$, $p$, $G$, $H$, $x$, $B$, $T$ from the introduction.
Recall that the canonical morphism $G/H_{\rm red}\to G/H$, $H_{\rm red}=H(k)$ the reduced scheme associated to $H$, is a homeomorphism.
See \cite[III \S3]{DG}. 
We will usually deal with varieties, unless we explicitly use the term ``scheme".
Schemes are always supposed to be algebraic over $k$. We will use ordinary notation for notions like inverse image and intersection,
and special notation for their schematic analogues.

\subsection{Frobenius splitting}
We give a very brief sketch of the basics of Frobenius splitting, introduced by Mehta and Ramanathan in \cite{MeRa}.
For precise statements we refer to \cite{BriKu}.
We start with the affine case. For a commutative ring $A$ of characteristic $p$ put $A^p=\{a^p\,|\,a\in A\}$, a subring of $A$.
First assume $A$ is reduced. Then one can define a Frobenius splitting of $A$ as an $A^p$-module direct complement to $A^p$ in $A$.
For example for a polynomial ring $A=k[x_1,\ldots,x_n]$, a Frobenius splitting is given by the $A^p$-span
of the restricted monomials (i.e. exponents $<p$) different from 1. Of course a Frobenius splitting of $A$ can also
be seen as an $A^p$-linear projection from $A$ onto $A^p$, i.e. as an $A^p$-linear map $\sigma:A\to A^p$ with $\sigma(1)=1$.
This is the point of view of \cite{Mat}. If we ``take the $p$-th root from this definition" we obtain the usual
definition: a Frobenius splitting of $A$ is a homomorphism of abelian groups $\sigma:A\to A$ with $\sigma(a^pb)=a\sigma(b)$ for all $a,b\in A$ and $\sigma(1)=1$.
This definition also works for non-reduced rings, and then it follows immediately that a Frobenius split ring is always reduced.
Finally one can ``sheafify" this definition to obtain the definition of Frobenius splitting for any scheme $X$ (separated and of finite type) over $k$:
Let $F:X\to X$ be the absolute Frobenius morphism, i.e. the morphism of schemes (not over $k$) which is the identity on the underlying topological space
and for which $F^\#:\mc O_X\to\mc O_X$ is the $p$-th power map.
Then a Frobenius splitting of $X$ is a morphism of $\mc O_X$-modules $\sigma:F_*(\mc O_X)\to\mc O_X$ which maps $1$ to $1$ (in $\mc O_X(X)$ and therefore in any $\mc O_X(U)$).
We can also consider $F^\#$ as a morphism of $\mc O_X$-modules $:\mc O_X\to F_*(\mc O_X)$ and then a Frobenius splitting of $X$ is
a left-inverse of the $\mc O_X$-module morphism $F^\#$. If $\sigma$ is a Frobenius splitting of $X$, then a closed subscheme of $X$ is called
{\it compatibly split} (with $\sigma$) if its ideal sheaf is stable under $\sigma$.
A compatibly split subscheme is itself Frobenius split.

From the existence of a Frobenius splitting one can deduce cohomology vanishing results.
For example, if $X$ is a Frobenius split variety which is proper over an affine, and $\mc L$ is an ample line bundle on $X$, then $H^i(X,\mc L)=0$ for all $i>0$.
See \cite[Thm.~1.2.8]{BriKu}. This is based on the fact that, in the presence of a Frobenius splitting, one can embed $H^i(X,\mc L)$ in $H^i(X,\mc L^{p^r})$ for all $r\ge1$.
Using compatible splittings or splitting relative to a divisor one can obtain more refined results, see \cite[Lem.~1.4.7, 1.4.11]{BriKu}.

An important result about Frobenius splittings of a smooth variety $X$ is that they correspond
to certain nonzero global sections of the $(1-p)$-th power of the canonical line bundle $\omega_X$,
see \cite[Thm.~1.3.8]{BriKu}. This result was first proved for $X$ smooth and projective in \cite[Sect.~2]{MeRa}.
Since $H^0(X,\omega_X^{1-p})$ may be zero, this also shows that, not every smooth variety is Frobenius split.

If $G$ acts on a variety $X$, then it also acts on $\Hom_{\mc O_X}(F_*(\mc O_X),\mc O_X)$ and turns this vector space into a rational $G$-module, see \cite[4.1.5]{BriKu}.
If $X$ is smooth  then the same holds for $H^0(X,\omega_X^{1-p})$.
Now one can define a Frobenius splitting to be {\it $B$-canonical} if it is $T$-fixed and is killed by all divided powers $e_\alpha^{(n)}\in\Dist(G)$, $\alpha$ simple relative to $B$, $n\ge p$.
Here $\Dist(G)$ is the {\it distribution algebra} or hyper algebra of $G$.
The notion of canonical splitting was introduced by Mathieu who proved the following result.
If $\mc L$ is a $G$-linearised line bundle on a $B$-canonically split $G$-variety $X$, then the $G$-module $H^0(X,\mc L)$ has a good filtration.
See \cite[Sect.~5,6]{Mat} and \cite[Ch.~4]{BriKu} for more detail. The notion of {\it good filtration} can be found in \cite{Jan}, \cite{Mat} or \cite{BriKu}.

We finish our discussion of Frobenius splittings, with a lemma about splittings of $G/P$.
The lemma seems to be known, but we provide a proof for lack of reference.
It will be needed in the proof of Proposition~\ref{prop.toroidal}.
For a parabolic subgroup $P$ containing $B$ we denote by $\rho_P$
the half sum of the roots of $T$ in the unipotent radical $R_uP$ and we put $\rho=\rho_B$.
We denote the opposite parabolic by $P^-$.
For $\lambda$ a character of $P^-$ (viewed as a character of $T$ orthogonal to the simple coroots of the Levi of $P$ containing $T$)
we denote the line bundle $G\times^{P^-}k_\lambda$ on $G/P^-$ by $\mc L(\lambda)$.
\begin{lem}\label{lem.GmodP}
Let $P$ be a parabolic subgroup of $G$. Then $G/P$ is split by a $(p-1)$-th power.
\end{lem}

\begin{proof}
We prove the claim for $G/P^-$. It is well known and easy to see that $\rho_P$ is dominant,
see e.g. \cite[Sect~3.1]{BriKu} ($\rho_P$ is there denoted $\delta_P$). So by \cite[Thm 3.1.2(c)]{BriKu}
the multiplication map $H^0(G/P^-,\mc L(2\rho_P))^{\ot(p-1)}\to H^0(G/P^-,\mc L(2(p-1)\rho_P))$ is surjective.
Now $\mc L(-2\rho_P)$ is the canonical bundle of $G/P^-$, so the assertion follows from \cite[Ex.~1.3.E(2)]{BriKu}. 
\end{proof}

\subsection{Spherical embeddings}
We recall some facts from the theory of spherical embeddings. For more detail we refer to \cite{Bri3}, \cite{Knop}, and \cite{Tim}.
A {\it $G/H$-embedding} is a normal $G$-variety $X$ with a $G$-equivariant embedding $G/H\hookrightarrow X$ of which the image is open and dense.
The set of $T$-weights of the nonzero $B$-semi-invariant rational functions on $G/H$ is denoted by $\Lambda=\Lambda_{G/H}$.
Since $B\cdot x$ is open in $G/H$ every $B$-semi-invariant rational function on $G/H$ is defined at $x$,
and completely determined by its $T$-weight and its value at $x$. So $k(G/H)^{(B)}/k^\times\cong\Lambda$, where $k(G/H)^{(B)}$ is the group of $B$-semi-invariant
rational functions and $k^\times$ is the multiplicative group of $k$. Put $\Lambda_{\mb Q}=\mb Q\ot_{\mb Z}\Lambda$. Then $\Hom_{\mb Z}(\Lambda,\mb Q)=\Lambda_{\mb Q}^*$.
Any (discrete) valuation of the function field $k(G/H)$ determines a $\mb Z$-linear function $k(G/H)^{(B)}/k^\times\to\mb Q$, i.e. an element of $\Lambda_{\mb Q}^*$.
In particular, any prime divisor (irreducible closed subvariety of codimension $1$) of any $G/H$-embedding determines an element of $\Lambda_{\mb Q}^*$.
A $G$-invariant valuation is completely determined by its image in $\Lambda_{\mb Q}^*$, so we can identify the set of $G$-invariant valuations of $G/H$
with a subset $\mc V=\mc V(G/H)$ of $\Lambda_{\mb Q}^*$. It is well-known that $\mc V$ is a finitely generated convex cone which spans $\Lambda_{\mb Q}^*$.

If $X$ is a $G/H$-embedding, then the complement of the open $B$-orbit has the $B$-stable prime divisors as its irreducible components. 
These we divide in two classes: The ones that are $G$-stable are called {\it boundary divisors} and will be denoted by $X_i$.
The others are the $B$-stable prime divisors that intersect $G/H$. These will be denoted by $D$, possibly with a subscript.
A $G/H$-embedding is called {\it simple} if it has a unique closed $G$-orbit. The {\it coloured cone} of a simple embedding $X$ is the pair $(\mc C,\mc F)$ where
$\mc F$, the set of colours of $X$, is the set of $B$-stable prime divisors of $G/H$ whose closure in $X$ contains the closed orbit of $X$ and $\mc C\subseteq\Lambda_{\mb Q}^*$
is the cone generated by the images of all boundary divisors of $X$ and the image of $\mc F$. A simple embedding is completely determined by
its coloured cone. For an arbitrary embedding each orbit is the closed orbit of a unique simple open sub-embedding.
The coloured cones of these sub-embeddings form the {\it coloured fan} of the embedding. An embedding is completely determined by its coloured fan.
More generally, coloured fans can be used to describe morphisms between embeddings of different homogeneous spaces, see \cite[Sect.~4]{Knop}.

A $G/H$-embedding is called {\it toroidal} if no $B$-stable prime divisor of $X$ which intersects $G/H$ (i.e. which is not $G$-stable) contains a $G$-orbit.
The toroidal embeddings are the embeddings corresponding to the coloured fans without colours, that is, the fans (in the sense of toric geometry) whose support is contained in $\mc V(G/H)$.
For a toroidal $G/H$-embedding $X$ we denote by $X_0$ the complement in $X$ of the union of the $B$-stable prime divisors that intersect $G/H$.
Note that every $G$-orbit in $X$ intersects $X_0$. We denote the closure of $T\cdot x$ in $X_0$ by $\ov{T\cdot x}^{X_0}$.
We will say that the {\it local structure theorem} holds for $X$ (relative to $x$, $T$ and $B$) if 

\medskip
{\it there exists a parabolic subgroup $P$ of $G$ containing $B$ such that, with $M$ the Levi subgroup of $P$ containing $T$, 
the variety $\ov{T\cdot x}^{X_0}$ is $M$-stable, the derived group $DM$ acts trivially on it, and the action of $G$
induces an isomorphism $R_uP\times \ov{T\cdot x}^{X_0}\stackrel{\sim}{\to}X_0$.}
\medskip

The local structure theorem for $X$ implies the local structure theorem for $G/H$ itself.
Furthermore, the parabolic subgroup $P$ only depends on $G/H$: it can be characterised as the stabiliser of the open $B$-orbit $B\cdot x$ or as the
stabiliser of the open subset $BH_{\rm red}$ of $G$ (for the action by left multiplication).

Whenever $M$ is a closed $k$-subgroup scheme of $G$ and $L$ a Levi subgroup of $G$ containing $T$ we denote by $M_L$ the schematic intersection of $L$ and $M$.
When the local structure theorem holds for $G/H$, then every rational function on $T\cdot x=T/H_T$ has a unique $U=R_uB$-invariant extension to $G/H$.
So $\Lambda=\mc X(T/H_T)$, the character group of $T/H_T$, and $\Lambda_{\mb Q}^*\cong\mb Q\ot_{\mb Z}\mc Y(T/H_T)$, where $\mc Y(T/H_T)$ is the cocharacter group of $T/H_T$.\footnote{
In \cite[Rem.~1.2]{T}, where this is also stated, it should have been assumed that the local structure theorem holds for $G/H$.}

\begin{lem}\label{lem.local_structure}
Let $X$ be a toroidal $G/H$-embedding for which the local structure theorem holds. 
\begin{enumerate}[{\rm(i)}]
\item Every $G$-orbit intersects $\ov{T\cdot x}^{X_0}$ in a $T$-orbit, and 
$\ov{T\cdot x}^{X_0}$ is a $T/H_T$-embedding whose fan is the same as that of $X$.
Furthermore, all $G$-orbit closures of $X$ are normal and the local structure theorem also holds for them (with the same $T,B,P$).
\item If $X$ is complete, then every closed $G$-orbit is isomorphic to $G/P^-$.
\end{enumerate}
\end{lem}
\begin{proof}
(i).\ The first assertion follows by the same arguments as in the proof of \cite[Thm.~29.1]{Tim} or in \cite[Sect.~2.4]{Bri3}.
The normality of the $G$-orbit closures follows from the toric case \cite[Prop.~I.2]{KKMS} and the final assertion is obvious.\\
(ii).\ Let $Y$ be a closed $G$-orbit in $X$ and let $X'$ be the corresponding simple open sub-embedding of $X$. Then $Y$ is complete.
So the cone in $\Lambda_{\mb Q}^*$ corresponding to $X'$ has dimension $\dim\Lambda_{\mb Q}^*=\rk\mc X(T/H_T)=\dim T/H_T$, by \cite[Thm.~6.3]{Knop}. 
So $\ov{T\cdot x}^{X'_0}$ has a unique $T$-fixed point $y$ which must be the intersection of $Y$ with $\ov{T\cdot x}^{X'_0}$. 
Let $P_0\le G$ be the stabiliser of $y$, then we have a bijective $G$-equivariant morphism $\varphi:G/P_0\to G\cdot y=Y$. 
By \cite[Prop.~AG.18.3(1)]{Bo}, $G/P_0$ is complete, i.e. $P_0$ is a parabolic.
Since $P_0$ contains $T$, we have that for every root of $T$ in $G$, at least one of the root subgroups $U_\alpha$, $U_{-\alpha}$ is contained in $P_0$.
It is then obvious from the local structure theorem for $Y$, that $P_0=P^-$ and that $\varphi$ is an isomorphism.
\end{proof}

A {\it wonderful compactification} of $G/H$ is a smooth simple complete toroidal $G/H$-embedding. Since it is is simple and toroidal it is determined by a cone in $\mc V$.
Since it is also complete the cone has to be all of $\mc V$, see \cite[Thm.~4.2]{Knop}. So a wonderful compactification is unique if it exists.
Furthermore, since cones in coloured fans of $G/H$-embeddings are always strictly convex, $\mc V$ contains no line. If $X$ is a wonderful compactification of $G/H$
for which the local structure theorem holds, then $\ov{T\cdot x}^{X_0}$ (and therefore $X_0$) is an affine space.

We will also need the following property of a smooth $G/H$-embedding $X$:

\medskip
{\it\noindent The anti-canonical bundle $\omega_X^{-1}$ has a nonzero $B$-semi-invariant global section of weight $2\rho_P$ with divisor $$\sum_Da_DD+\sum_iX_i\,,\eqno{(*)}$$
where the first sum is over the $B$-stable prime divisors of $X$ that are not $G$-stable
with the $a_D$ strictly positive and the second sum is over the boundary divisors of $X$.}
\medskip

If $X$ is a toroidal smooth $G/H$-embedding which satisfies the local structure theorem, then it is easy to see
that $\omega_X^{-1}$ has a nonzero $B$-semi-invariant global section of weight $2\rho_P$, unique up to scalar multiples,
and that its divisor is given by the formula above with the $a_D$ integers.
So the crucial point of (*) is that the coefficients $a_D$ are strictly positive.

\begin{lem}[{cf. proof of \cite[Prop.~4.1]{Bri1}}]\label{lem.candiv}
Let $X$ be a smooth $G/H$-embedding. Assume the local structure theorem holds for the
toroidal embedding obtained by removing all orbits of codimension $>1$ from $X$,
and assume the $T$-orbit map of $x$ is separable, then $X$ has property (*).
\end{lem}

\begin{proof}
Since the proof in \cite{Bri1} is very brief we give a bit more detail. We reduce as in [loc. cit.] to the case that $X$ is smooth and toroidal
by removing the orbits of codimension $>1$. Then we have the local structure theorem $R_uP\times \ov{T\cdot x}^{X_0}\stackrel{\sim}{\to}X_0$.
If an algebraic group $K$ acts on $X$ by automorphisms, then $\Lie(K)$ acts on $\mc O_X$ by derivations, put differently, we have a Lie algebra homomorphism
$\varphi:\Lie(K)\to\Vect(X)$. Furthermore, 
$$\varphi(u)_y\in T_y(K\cdot y) \eqno(\dagger)$$
for all $u\in\Lie(K)$ and $y\in X$ and $T_y(-)$ is the Zariski tangent  space at $y$. We have
$\dim(T_y(B\cdot y))=\dim(B\cdot y)<\dim(B\cdot x)=\dim(R_u P)+\dim(T\cdot x)$ for all $y\in X\sm B\cdot x$. Let $\theta$ be the wedge product of
the $\varphi$-images of a basis of $\Lie(R_uP)$ together with a lift to $\Lie(T)$ of a basis of $\Lie(T\cdot x)$.
Here the lift of the basis of $\Lie(T\cdot x)$ to elements in $\Lie(T)$ exists by the separability assumption.
Clearly $\theta$ is $B$-semi-invariant of weight $2\rho_P$.
By ($\dagger$) applied to $K=B$ we have that $\theta$ vanishes on the complement
of the open $B$-orbit $B\cdot x=R_u(P)T\cdot x$, i.e. on the union of $B$-stable prime divisors.
Furthermore, it is well known from toric geometry, using the local structure theorem, that it vanishes to the order one on
the boundary divisors and it is easy to see that it is nonzero on $R_u(P)T\cdot x$.
So $\theta$ is a global section of the anti-canonical bundle with divisor
of zeros $\sum_Da_DD+\sum_iX_i$, $a_D>0$.
\end{proof}

For any smooth $G/H$-embedding $X$ with boundary divisors $X_i$ we put $\tilde\omega_X=\omega_X\ot\mc O_X(\sum_iX_i)$.
Like $\omega_X$, the sheaf $\tilde\omega_X$ comes with a canonical $G$-linearisation.
When $X$ is also toroidal $\tilde\omega_X$ is the log canonical sheaf. Note that if $X$ has property (*), then we have $\tilde\omega_X=\mc O_X(-\sum_Da_DD)$.

Let $H_{\rm a}$ (``$\rm a$" for adjoint, see Section~\ref{s.symmetric}) be a closed subgroup scheme of $G$ containing $H$ and
let $x_{\rm a}\in G/H_{\rm a}$ be the image of $e$ under the canonical map $G\to G/H_{\rm a}$.
From the next section on we will make the following assumptions; see Remark~\ref{rems.candiv}.1 for further comments.

\smallskip
\begin{enumerate}[({\rm A}1)]
\item $H_{\rm a}$ is generated by $H$ and a closed subgroup scheme of $T$ that normalises $H$;
\item the homogeneous space $G/H_{\rm a}$ has a wonderful compactification $\bf X$ for which the local structure theorem (see Section~\ref{s.prelim}) holds;
\item the wonderful compactification $\bf X$ of $G/H_{\rm a}$ has property (*) above;
\item the characters through which $H$ acts on the top exterior powers of $T_x(G/H)$ and $T_{x_{\rm a}}(G/H_{\rm a})$ are the same.
\end{enumerate}

From (A1) and (A2) it follows easily that the local structure theorem holds for all toroidal $G/H$-embeddings. One just has to use the fact that for every toroidal
$G/H$-embedding $X$ there is a unique $G$-equivariant morphism $X\to\bf X$ which maps $x$ to $x_{\rm a}$, see \cite[Thm.~4.1]{Knop}.
This was also observed (in more specific situations) in \cite[Prop.~6.2.3(i)]{BriKu} and \cite[Prop.~3.2(i)]{T}.

Concerning (A4), to see that $H$ acts on $T_x(G/H)$ it is easiest to use the ``functor of points" approach and dual numbers. For a commutative $k$-algebra $R$ put $R[\delta]=R[t]/(t^2)$,
where $t$ is an indeterminate and $\delta=t+(t^2)$. Let $X$ be a smooth scheme over $k$.
Then we can define the ``tangent bundle functor" $$T(X):\{\text{commutative $k$-algebras}\}\to\{\text{sets}\}$$ as follows: $T(X)(R)=X(R[\delta])$.
Furthermore, we obtain a morphism $:T(X)\to X$ of functors given by the homomorphism $R[\delta]\to R$ which maps $\delta$ to $0$.
Let $x\in X(k)$. Then the fiber over $x$ is $T_x(X)$, see \cite[Cor.~II.4.3.3]{DG} or \cite[AG.16.2]{Bo}.
If $H$ acts on $X$, see \cite[Sect.~II.1.3]{DG} or \cite[I.2.6]{Jan}, then it also acts on $T(X)$ and the above morphism is $H$-equivariant. So if $H$ fixes $x$,
then $H$ acts on $T_x(X)$. This action is clearly linear, so $H$ will then also act on the top exterior power of $T_x(X)$, see \cite[I.2.7(3)]{Jan}.

Note that if $f:G\to G'$ is a central (cf. \cite[V.2.2]{Bo}) surjective morphism of reductive groups, $H'\subseteq H'_{\rm a}$ closed subgroup schemes of $G'$
and $H,H_{\rm a}$ the schematic inverse image of $H',H'_{\rm a}$ under $f$,
then $G,H,H_{\rm a}$ satisfy (A1)-(A4), if $G',H',H'_{\rm a}$ do (relative to $B'=f(B)$, $T'=f(T)$ etc.).
We denote the boundary divisors of $\bf X$ by ${\bf X}_1,\ldots,{\bf X}_r$.
Note that by the adjunction formula (see e.g. \cite[Prop.~II.8.20]{H2}) the restriction of $\tilde\omega_{\bf X}$ to $G/P^-$ is isomorphic to $\omega_{G/P^-}$.

\begin{lem}\label{lem.pull-back}
Assume (A1)-(A4) and let $X$ be a smooth toroidal $G/H$-embedding. Consider the morphism $\pi:X\to\bf X$ which extends the canonical map $G/H\to G/H_{\rm a}$
Then $\tilde\omega_X$ is $G$-equivariantly isomorphic to the pull-back $\pi^*\tilde\omega_{\bf X}$.
In particular, all smooth $G/H$-embeddings have property (*).
\end{lem}
\begin{proof}
Let $s_1$ and $s_2$ be the (unique up to scalar multiples) nonzero $B$-semi-invariant rational section of $\tilde\omega_X$ resp. $\tilde\omega_{\bf X}$ of
weight $-2\rho_P$. They are the nonzero $B$-semi-invariant rational sections of $\omega_X$ and $\omega_{\bf X}$ multiplied with 
canonical sections of the boundary divisors of $X$ resp. ${\bf X}$.
Let $s_2'$ be the pull-back of $s_2$ to $\pi^*\tilde\omega_{\bf X}$.
First we show that the divisor $(s_1)$ equals $(s_2')$. From the local structure theorems for $X$ and ${\bf X}$ it is clear that $(s_1)$ and $(s_2')$
do not involve the boundary divisors of $X$, so it is enough to show that $(s_1)$ and $(s_2')$ are equal on $G/H$.
Since the equivariant Picard group ${\rm Pic}_G(G)$ is trivial $s_1$ and $s_2'$ pull back to $(B\times H)$-semi-invariant functions on $G$.
Since they have the same $B$-weight $-2\rho_P$ it is enough to show they have the same $H$-weight and this follows from (A4).

Since $(s_1)$ equals $(s_2')$, there is an isomorphism between $\tilde\omega_X$ and $\pi^*\tilde\omega_{\bf X}$ such that $s_1$ corresponds to $s_2'$.
Since any two $G$-linearisations of a line bundle on $X$ differ by a character of $G$ and $s_1$ and $s_2'$ have the same weight
it is clear that this isomorphism is $G$-equivariant.

The final assertion follows by reducing to the case that $X$ is (smooth and) toroidal by removing the orbits of codimension $>1$.
\end{proof}

Note that if we define a canonical divisor of a $G/H$-embedding $X$ to be a divisor whose restriction to the smooth locus $X'$ of $X$
is a canonical divisor of $X'$, then, assuming (A1)-(A4), a canonical divisor of any $G/H$-embedding $X$ is given by the formula in (*).

\begin{rems}\label{rems.candiv}
1.\ Of the assumptions (A1)-(A4), (A2) is the most important one; it is for example not satisfied by the varieties of unseparated flags from \cite{HL}.
It is essential that in the local structure theorem the isomorphism is given by the action of $G$.
If $H$ and $H_{\rm a}$ are reduced we can deduce (A4) as follows. We have canonical isomorphisms $T_x(G/H)\cong\Lie(G)/\Lie(H)$ and
$T_{x_{\rm a}}(G/H)\cong\Lie(G)/\Lie(H_{\rm a})$ of $H$-modules and we obtain an exact sequence of $H$-modules
$$0\to\Lie(H_{\rm a}/H)\to T_x(G/H)\to T_{x_{\rm a}}(G/H_{\rm a})\to0\,,$$
where we used that, by (A1), $H$ is normal in $H_{\rm a}$.
So the top exterior power of $T_x(G/H)$ is the tensor product of the top exterior powers of $T_{x_{\rm a}}(G/H_{\rm a})$ and $\Lie(H_{\rm a}/H)$.
The action of $H$ on $\Lie(H_{\rm a}/H)$ is trivial, since the $H$-action on $H_a/H$ induced by the conjugation action of $H$ on $H_a$ is trivial. Thus (A4) follows.
Now (A4) holds in characteristic $0$ by Cartier's Theorem \cite[II.6.1.1]{DG}.
I do not know of any examples where (A1) and (A2) are satisfied, but not (A3) or (A4).\\
2.\ In \cite{Bri1} (see also \cite[Prop.~3.6]{Luna}) the coefficients $a_D$ from (*) are determined and, at least in many special cases, this can also be done in prime characteristic.
For example, when $p>2$ and $G/H$ is a symmetric space $G/G^{\theta}$, $G$ adjoint semi-simple (see also Section~\ref{s.symmetric}), then this amounts to writing the canonical
divisor plus the sum of the boundary divisors as a linear combination of the images in ${\rm Pic}({\bf X})$ of the $B$-stable prime divisors of ${\bf X}$ that are not $G$-stable,
where ${\bf X}$ is the wonderful compactification of $G/H$ (these form basis of ${\rm Pic}({\bf X})$).
Since ${\rm Pic}({\bf X})$ embeds naturally in ${\rm Pic}(G/P^-)$, $G/P^-$ the closed orbit, via restriction (see \cite{DeC-Spr}), and the line bundle corresponding to the above divisor
restricts to the canonical bundle of ${\rm Pic}(G/P^-)$ this amounts to writing $-2\rho_P$
as a linear combination of the above images. In the notation of \cite[Prop 26.22]{Tim} these images are: the $\hat{\varpi}_i$, $i$ a $\iota$-fixed node of the Satake diagram,
and the $\varpi_j+\varpi_{\iota(j)}$, $j$ a $\iota$-unstable node of the Satake diagram, except that we split the weights in the second list
that correspond to the exceptional simple roots (see \cite[Sect.~4]{DeC-Spr}) into their two summands.\\
3.\ In characteristic $0$ all orbit closures in $G/H$-embeddings are normal, see \cite[Prop.~3.5]{BriP}.
The following example in characteristic $p$ was mentioned to me by M.~Brion. Put $G=\SL_2\times k^\times\times k^\times$, $V=k^2$, and $W=V\oplus V$.
Let $G$ act on $W$ by $(g,\lambda,\mu)\cdot (v,w)=(\lambda\,g\cdot v,\mu {\rm Fr}(g)\cdot w)$, where ${\rm Fr}$ is the Frobenius morphism which raises all matrix entries to the power $p$.
Now let $\la-,-\ra$ be the symplectic form on $V$ given by $\la v,w\ra=v_1w_2-v_2w_1$. This form is preserved by $\SL_2$.
Now define $f\in k[W]$ by $f(v,w)=\la {\rm Fr}(v),w\ra$, where ${\rm Fr}$ is the Frobenius morphism of $k^2$. Then $W$ is spherical and the zero locus of $f$
is a non-normal orbit closure, since its singular locus has codimension $1$. 
See \cite[Cor to Prop.~2.1]{T} for a result that guarantees the normality of the $G$-orbit closures under a certain assumption.
\end{rems}

\section{Frobenius splittings and rational resolutions}\label{s.frobsplit_ratres}
We retain the notation from the previous sections and we assume that properties (A1)-(A4) hold. We will consider the following further assumptions

\smallskip
\begin{enumerate}[({\rm B}1)]
\item the restriction map $H^0({\bf X},\tilde\omega_{\bf X}^{-1})\to H^0(G/P^{-},\omega_{G/P^{-}}^{-1})$ is surjective;\medskip
\item the restriction map $H^0({\bf X},\tilde\omega_{\bf X}^{1-p})\to H^0(G/P^{-},\omega_{G/P^{-}}^{1-p})$ is surjective and its kernel has a good filtration.
\end{enumerate}

Recall that a line bundle is called {\it semi-ample} if some positive power of it is generated by its global sections.
\begin{prop}\label{prop.toroidal}
Assume (B1) and let $X$ be a smooth toroidal $G/H$-embedding. Then
\begin{enumerate}[{\rm(i)}]
\item $X$ is Frobenius split by a $(p-1)$-th power which compatibly splits all $G$-stable closed subvarieties.
\item Assume $X$ is projective over an affine, let $\mc L$ be a semi-ample line bundle on $X$ and let $Y$ be a $G$-stable closed subvariety of $X$.
\begin{enumerate}[{\rm(a)}]
	\item $H^i(Y,\mc L)=0$ for all $i\ge 1$.
	\item If $X$ is simple or $Y$ is irreducible, then the restriction map $H^0(X,\mc L)\to H^0(Y,\mc L)$ is surjective.
\end{enumerate} 
\end{enumerate}
\end{prop}

\begin{proof}
(i).\ The local structure theorem holds for all toroidal $G/H$-embeddings, so, by arguing as at the beginning of \cite[Sect.~29.2]{Tim},
$X$ has a toroidal smooth completion. So we may assume $X$ complete as well.
Let $\pi:X\to\bf X$ be as in Lemma~\ref{lem.pull-back}. By Lemma~\ref{lem.local_structure} its restriction to any closed orbit is an isomorphism.
By Lemma~\ref{lem.GmodP} there exists $s\in H^0(G/P^{-},\omega_{G/P^{-}}^{-1})$ with $s^{p-1}$ a splitting of $G/P^-$.
By assumption we can find a lift $\ov s\in H^0({\bf X},\tilde\omega_{\bf X}^{-1})$ of $s$.
Note that by Lemma~\ref{lem.pull-back} we have $\pi^*(\tilde\omega_{\bf X})=\tilde\omega_X$.
Let $\sigma_i$ be the canonical section of $\mc O_X(X_i)$ and put $\tau=\pi^*(\ov s)\prod_i\sigma_i\in H^0(X,\omega_X^{-1})$.
Then $\tau^{p-1}\in H^0(X,\omega_X^{1-p})$ is a splitting of $X$ which compatibly splits the boundary divisors $X_i$
and therefore all $G$-stable closed subvarieties. This follows by considering its local expansion at a point in a closed orbit.
See \cite[Thm.~2]{Rit}, \cite[Thm.~3.1]{Str} or \cite[Thm.~6.1.12, Ex.~1.3.E.4]{BriKu} for more details.\\
(ii)(a).\ Let $\ov s$, the $\sigma_i$ and $\tau$ be as in (i) (restricted to the original $X$) and put $E=(\pi^*(\ov s))$, the divisor of $\pi^*(\ov s)$.
By \cite[Prop.~1.3.11]{BriKu} $E$ is reduced and effective and contains none of the $X_i$, and the splitting given by $\tau$ is compatible with $(\tau)=E+\sum_iX_i$.
By \cite[Lem.~1.4.11]{BriKu} there is for all $j$, $0\le n_j,m<p^r$ and $r\ge0$ a split injection
$$H^i(X,\mc L)\to H^i(X,\mc L^{p^r}(\sum_jn_jX_j+mE))\,.\eqno(_*^*)$$ 
Now $E$ is linearly equivalent to the divisor $\sum_Da_DD$, the $a_D$ as in ($*$), see Lemma~\ref{lem.pull-back}.
Since $X$ is quasiprojective it has a $B$-stable ample effective divisor $E'$ (cf. proof of \cite[Cor.~6.2.8]{Bri3}). We can write $E'=\sum_jn'_jX_j+\sum_Db_DD$.
If we choose $n_j=n_j'$ and $m$ such that $ma_D\ge b_D$ for all $D$, then $\sum_jn_jX_j+mE \sim\sum_jn'_jX_j+\sum_Dma_DD$ will be ample,
since each $\mc O_X(D)$ is generated by its global sections. The latter follows from the fact that no $D$ contains a $G$-orbit, since $X$ is toroidal.
See the argument in the proof of \cite[Prop.~2.2]{Bri3}. Finally, we choose $r$ such that $n_j,m<p^r$. Now the assertion for $Y=X$ follows from ($_*^*$) and \cite[Thm.~1.2.8]{BriKu}.
The case of arbitrary $Y$ follows by observing that $\tau$ restricts to a splitting of $Y$ and that $Y$ is a smooth toroidal spherical variety
which is projective over an affine. Furthermore, the local structure theorem clearly holds for $Y$. So we can apply the same arguments as above.\\
(b).\ First assume $X$ is simple. Then $E\sim \sum_Da_DD$ is ample by \cite[Cor.~17.20]{Tim} and it is compatibly split by $\tau^{p-1}$.
So $\tau^{p-1}$ is an $(p-1)E$-splitting and therefore also an $E$-splitting by \cite[Thm.~1.4.10, Rem.~1.4.2(ii)]{BriKu}.
Furthermore, $E$ contains no $G$-orbit, since $\pi^*(\ov s)$ is nonzero on the closed orbit. So all closed $G$-stable
subvarieties are compatibly $E$-split and the result follows from \cite[Thm.~1.4.8(ii)]{BriKu}.

Now assume $Y$ is irreducible. Then $X$ has an ample $G$-equivariant line bundle which, by \cite[Cor.~2.3]{Knop2},
has a $B$-semi-invariant global section which is nonzero on $Y$. Take $E'$ to be the divisor of this section.
Then $E'$ is $B$-invariant, ample and doesn't contain $Y$. Now we proceed as in the proof of (a) where we have that
$n_j=n_j'=0$ whenever $Y\subseteq X_j$. Since $Y$ is compatibly $(p-1)\big(E+\sum_{X_j\nsupseteq Y}X_j\big)$-split we can combine the arguments from Lemmas 1.4.7 and 1.4.11 in \cite{BriKu}.
We form a commutative square
$$
\xymatrix{
H^i(X,\mc L)\ar[r]\ar[d]& H^i\big(X,\mc L^{p^r}(\sum_jn_jX_j+mE)\big)\ar[d]\\
H^i(Y,\mc L)\ar[r]& H^i\big(Y,\mc L^{p^r}(\sum_jn_jX_j+mE)\big)\quad,
}
$$
where the top horizontal row is ($_*^*$), the bottom horizontal row is ($_*^*$) with $X$ replaced by $Y$, and the vertical arrows are restriction maps.
Then the splittings of the horizontal arrows are compatible, i.e. they form a commutative square with the vertical arrows.
Since the right vertical arrow is surjective, it follows that the same holds for the left vertical arrow.
This is just a consequence of the fact that if a homomorphism of abelian groups $$f:M=M_1\oplus M_2\to N=N_1\oplus N_2$$
with $f(M_i)\subseteq N_i$, $i\in\{1,2\}$, is surjective, then the restrictions
$f|_{M_1}:M_1\to N_1$ and $f|_{M_2}:M_2\to N_2$ are surjective.
\end{proof}

As in \cite[Sect.~3.3]{BriKu} (following Kempf) we define a morphism $f:X\to Y$ of varieties to be {\it rational} if the direct image under $f$ of the structure sheaf of $X$ is that of $Y$ and if the higher direct images are zero, that is, if $f_*(\mc O_X)=\mc O_Y$ and $R^if_*(\mc O_X)=0$ for $i>0$.
Recall that a {\it resolution (of singularities)} of an irreducible variety $X$ is a smooth irreducible variety $\tilde X$ together with a proper birational morphism $\varphi:\tilde X\to X$. Note that if $X$ is normal we have $\varphi_*(\mc O_{\tilde X})=\mc O_X$.
A {\it rational resolution} is a resolution $\varphi:\tilde X\to X$ which is a rational morphism and satisfies $R^i\varphi_*\omega_{\tilde X}=0$ for all $i>0$.

The following lemma was implicit in the proof of \cite[Cor.~6.2.8]{BriKu}.
\begin{lem}\label{lem.vanishing}\
Assume that every projective $G/H$-embedding $X$ has a $G$-equi-variant resolution $\psi:\tilde X\to X$ with $\psi$ projective and $\tilde X$ toroidal
and assume that every such resolution is a rational morphism. Then every projective birational $G$-equivariant morphism $\varphi:\tilde X\to X$ is rational
for every $G/H$-embedding $X$.
\end{lem}
\begin{proof}
Let $X$ be a $G/H$-embedding and let $\varphi:\tilde X\to X$ be $G$-equivariant, projective and birational.
By the Sumihiro Theorem \cite{Sum1},\cite{Sum2} we may assume that $X$ is quasi-projective. 
Then $X$ and $\tilde X$ embed as open subsets in projective $G/H$-embeddings $Y$ and $\tilde Y$. 
Now let $Z$ be the normalisation of the closure of the graph of $\varphi$ in $\tilde Y\times Y$. Then $Z$ is projective and $\tilde X$ identifies with a $G$-stable open subset of $Z$. Furthermore, the natural morphism $Z\to Y$ extends $\varphi$ and we denote it again by $\varphi$. Note that $\varphi^{-1}(X)=\tilde X$. Now let $\psi:\tilde Z\to Z$ be a resolution as above.
Then $\tilde Z$ is projective. Since $\psi$ is a rational morphism, the Grothendieck spectral sequence for $\varphi_*$, $\psi_*$ and $\mc O_{\tilde Z}$ collapses and we obtain $(R^i\varphi_*)(\mc O_Z)=(R^i\varphi_*)(\psi_*\mc O_{\tilde Z})=R^i(\varphi\circ\psi)_*(\mc O_{\tilde Z})$. So the rationality of the morphism $\varphi$ follows from that of $\varphi\circ\psi$.
\end{proof}

\begin{thm}\label{thm.resolution}
Assume (B1) and let $X$ be a $G/H$-embedding. Then 
\begin{enumerate}[{\rm(i)}]
\item $X$ has a $G$-equivariant rational resolution $\varphi:\tilde X\to X$ with $\varphi$ projective and $\tilde X$ quasiprojective toroidal.
\item $X$ is Frobenius split compatible with all $G$-stable closed subvarieties.
\item Assume $X$ is proper over an affine, let $\mc L$ be a semi-ample line bundle on $X$ and let $Y$ be an irreducible $G$-stable closed subvariety of $X$.
\begin{enumerate}[{\rm(a)}]
	\item $H^i(X,\mc L)=0$ for all $i\ge 1$.
	\item The restriction map $H^0(X,\mc L)\to H^0(Y,\mc L)$ is surjective.
\end{enumerate}
\item If $Y$ is a scheme which is projective over an affine and $\varphi:X\to Y$ is a proper morphism, then $R^i\varphi_*(\mc O_X)=0$ for all $i>0$.
\end{enumerate}
\end{thm}

\begin{proof}
(i). We construct a resolution $\varphi:\tilde X\to X$ with $\varphi$ projective and $\tilde X$ smooth, quasiprojective and toroidal as in \cite[Prop.~6.2.5]{BriKu}, \cite[Prop.~3]{Rit}. 
First we consider the normalisation $X'$ of the closure in $X\times\bf X$ of the graph of $G/H\to G/H_{\rm a}$. In the language of coloured fans this means that we form a toroidal covering of $X$
by ``removing the colours" from the fan of $X$. Then we construct the desingularisation $\tilde X\to X'$ by simplicial subdivision of the fan of $X'$. The ``toric slice"
in $\tilde X_0$ has the same fan as $\tilde X$. So, by the local structure theorem, we deduce the smoothness of $\tilde X$ from that of the toric slice.
By the local structure theorem we can also deduce the quasi-projectivity of $\tilde X$ from that of the toric slice as in \cite[Prop.~6.2.3(iv)]{BriKu}:
Since the toric slice is quasiprojective, a nonnegative combination of the boundary divisors of the toric slice is ample.
Now we form the same combination of the corresponding boundary divisors of $\tilde X$. This is ample on $\tilde  X_0$ which is the inverse image of ${\bf X}_0$ under $\pi:\tilde X\to{\bf X}$.
By the $G$-equivariance of $\pi$ we obtain the this divisor is ample relative to $\pi$. So, since ${\bf X}$ is projective, it follows that $\tilde X$ is quasi-projective, see \cite[Prop. 4.6.13(ii)]{Gro}.
Alternatively one can deduce the quasi-projectivity of $\tilde X$ from the description of ample divisors on a spherical variety, see \cite[Sect.~17.5]{Tim}, \cite[Sect.~5]{Bri3}.

We show that $\varphi$ is rational. The arguments are very similar to the proof of \cite[Cor.~6.2.8]{BriKu}. Only for the vanishing of the higher direct images of
$\mc O_{\tilde X}$ we have to use Proposition~\ref{prop.toroidal}(ii)(a) rather than the arguments in [loc. cit.].
Since $\varphi$ is proper and birational and $X$ is normal, we have $\varphi_*(\mc O_{\tilde X})={\mc O}_X$. 
Let $\tau^{p-1}$ be the Frobenius splitting of $\tilde X$ given by Proposition~\ref{prop.toroidal}(i). Then $\tau$ vanishes on the exceptional locus of $\varphi$, since the latter is contained in
the complement of the open $G$-orbit, i.e. the union of the boundary divisors. So, by \cite[Thm.~1.3.14]{BriKu} we have $R^i\varphi_*(\omega_{\tilde X})=0$ for all $i\ge 1$.

It remains to show that $R^i\varphi_*(\mc O_{\tilde X})=0$ for all $i\ge 1$. By Lemma~\ref{lem.vanishing} we may assume that $X$ is projective. 
By Kempf's Lemma, \cite[Lem.~1]{Kempf} or \cite[Lem.~3.3.3(a)]{BriKu},
it suffices to show that there exists an ample line bundle $\mc L$ on $X$ such that $H^i(\tilde X,\varphi^*(\mc L))=0$ for all $i\ge 1$.
Here we note that for Kempf's Lemma it is enough that $\varphi$ is proper and $X$ proper over an affine. 
Note furthermore that $\varphi^*(\mc L)$ is semi-ample if $\mc L$ is, so by Proposition~\ref{prop.toroidal}(ii)(a) any ample line bundle $\mc L$ on $X$ will do, and there exists one, since $X$ is projective.\\
(ii).\ We take a resolution $\varphi:\tilde X\to X$ as in (i). Then, by \cite[Lem.~1.1.8]{BriKu}, we can push a splitting of $\tilde X$ as in Proposition~\ref{prop.toroidal}(i) forward to $X$,
since $\varphi_*(\mc O_{\tilde X})={\mc O}_X$. It will compatibly split all $G$-stable closed subvarieties, since this holds for the splitting of $\tilde X$.\\
(iii).\ Again we take a resolution $\varphi:\tilde X\to X$ as in (i). Note that $\tilde X$ is proper over an affine, since $X$ is.
Since $\tilde X$ is also quasi-projective, it is projective over an affine. Since $\varphi$ is rational $H^i(X,\mc L)\stackrel{\sim}{\to}H^i(\tilde X,\varphi^*(\mc L))$ for all $i\ge0$, see e.g. \cite[Lemma~3.3.2]{BriKu}.
Choose $\tilde Y\subseteq\tilde X$ irreducible $G$-stable with $\varphi(\tilde Y)=Y$. Then the pull-back map $H^0(Y,\mc L)\to H^0(\tilde Y,\varphi^*(\mc L))$ is injective.
Now $\varphi^*(\mc L)$ is semi-ample since $\mc L$ is, so the assertions follow from Proposition~\ref{prop.toroidal}(ii).\\
(iv). This follows from (iii) and Kempf's Lemma, see the proof of (i).
\end{proof}

\begin{thm}\label{thm.cansplit}
Assume (B2). Let $X$ be a $G/H$-embedding. Then $X$ has a $B$-canonical Frobenius splitting which compatibly splits all $G$-stable closed subvarieties.
\end{thm}

\begin{proof}
After replacing $G$ by the connected centre times the simply connected cover of the derived group $DG$, we may assume that $\rho\in X(T)$. 
Let $\varphi:\tilde X\to X$ be a resolution as in Theorem~\ref{thm.resolution}(i).
Recall that the (first) Steinberg module ${\rm St}$ is irreducible and also isomorphic to the induced module
$\nabla((p-1)\rho)$ and to the Weyl module $\Delta((p-1)\rho)$, see \cite[II.3.18,19]{Jan}. Let $v_-$ and $v_+$ be a (nonzero) lowest and highest weight vector of ${\rm St}$.

If $X'$ is a smooth $G/H$-embedding, then, by \cite[Lem.~4.1.6]{BriKu}, $\sigma\in H^0(X',\omega_{X'}^{1-p})$ is $B$-canonical if an only if it is the image of $v_-\ot v_+$
under a $G$-module homomorphism ${\rm St}\ot{\rm St}\to H^0({X'},\omega_{X'}^{1-p})$. Let $s\in H^0(G/P^{-},\omega_{G/P^{-}}^{1-p})$ be the $B$-canonical splitting of $G/P^-$,
see \cite[Thm.~4.1.15]{BriKu}. Since ${\rm St}$ is a Weyl module, ${\rm St}\ot{\rm St}$ has a Weyl filtration by \cite[Prop.~II.4.21]{Jan} or \cite[Cor.~4.2.14]{BriKu}.
From \cite[Prop.~II.4.13]{Jan} one now easily deduces that the functor $\Hom_G({\rm St}\ot{\rm St},-)$ is exact on short exact sequences of $G$-modules with a good filtration.
So from (B2) it follows that the map
$$\Hom_G\big({\rm St}\ot{\rm St},H^0({\bf X},\tilde\omega_{\bf X}^{1-p})\big)\to\Hom_G\big({\rm St}\ot{\rm St},H^0(G/P^{-},\omega_{G/P^{-}}^{1-p})\big)$$ is surjective.
Now let $f\in\Hom_G({\rm St}\ot{\rm St},H^0(G/P^{-},\omega_{G/P^{-}}^{1-p}))$ be such that $f(v_-\ot v_+)=s$. We lift $f$ to $\ov f\in\Hom_G({\rm St}\ot{\rm St},H^0({\bf X},\tilde\omega_{\bf X}^{1-p}))$
and we put $\ov s=\ov f(v_-\ot v_+)$. Clearly $\ov s$ restricts to $s$.

Let $\pi:{\tilde X}\to\bf X$ be as in Lemma~\ref{lem.pull-back}.
As in the proof of Proposition~\ref{prop.toroidal}(i) one can show now, by considering a toroidal smooth completion of $\tilde X$,
that $\pi^*(\ov s)\prod_i\sigma_i^{p-1}\in H^0({\tilde X},\omega_{\tilde X}^{1-p})$
is a splitting of $\tilde X$ which compatibly splits all $G$-stable closed subvarieties. Clearly this splitting is $B$-canonical.
Finally we push the splitting down to $X$ by means of $\varphi$ and apply \cite[Lem.~1.1.8, Ex.~4.1.E(3)]{BriKu}.
\end{proof}

\begin{rems}\label{rem.cansplit}
1.\ The existence of a $B$-canonical splitting as in Theorem~\ref{thm.cansplit} implies existence of good filtrations in several situations,
see \cite[Thm.~4.2.13, Ex.~4.2.E(2)]{BriKu}. It also implies normality of the $G$-orbit closures, see \cite[Cor to Prop.~2.1]{T}.\\
2.\ Assumption (B2) is in fact equivalent to the existence of a $B$-canonical splitting of $\bf X$ which compatibly splits the closed orbit $G/P^-$.
Indeed, assuming the latter, we have that $H^0({\bf X},\tilde\omega_{\bf X}^{1-p})\to H^0(G/P^{-},\omega_{G/P^{-}}^{1-p})$ is surjective, since $\tilde\omega_{\bf X}^{1-p}$ is ample
by \cite[Cor.~17.20]{Tim} and $G/P^-$ is compatibly split, see \cite[Thm.~1.4.8]{BriKu}.
Furthermore, the kernel of the above map has a good filtration by \cite[Ex.~4.2.E(2)]{BriKu}.
Similarly, (B1) is equivalent to the existence of a splitting of $\bf X$ which compatibly splits $G/P^-$.
In particular, (B2) implies (B1).\\
3.\ Theorem~\ref{thm.resolution} does not hold for all equivariant embeddings of spherical homogeneous spaces.
The following example was mentioned to me by one of the referees. Let $Y$ be a $G$-variety of unseparated
flags admitting an ample line bundle $\mc L$ such that $H^1(Y,\mc L)\ne0$ (examples of such
varieties are given in \cite{HL}, Examples 3 and 4 in Section 6). Let $\pi:\tilde X\to Y$ be the structure
map of the dual line bundle, put $X := \Spec(k[\tilde X])$, the ``affine cone over $Y$", and let $\varphi:\tilde X\to X$ be
the canonical map. Then $X$ is normal and $\tilde X$ and $X$ are spherical $G\times k^\times$\,-varieties. Furthermore, $\varphi$ is an equivariant
resolution, see \cite[Lem.~1.1.13]{BriKu}. But $\varphi$ is not rational, since
$$H^0(X,R^1\varphi_*\mc O_{\tilde X}) = H^1(\tilde X,\mc O_{\tilde X}) = H^1(Y,\pi_*\mc O_{\tilde X}) = \bigoplus_{n\ge 0}H^1(Y,\mc L^n)\ne0\,.$$
Here the first equality holds since $X$ is affine, and the second one since $\pi$ is affine. 
See \cite[Prop.~III.8.5, Ex.~III.8.2]{H2}.
\end{rems}

\section{Symmetric spaces}\label{s.symmetric}
We notice that the case of group embeddings in any characteristic is already treated in \cite{Rit, T}.
In the remainder of this section we assume that ${\rm char}(k)\ne 2$. For background on symmetric spaces we refer to \cite{Rich, V, DeC-Spr, Tim}.
Let $\theta_{\rm ad}\ne\id$ be an involution of the adjoint group $G_{\ad}$ of $G$, let $\pi:G\to G_{\ad}$ be the canonical homomorphism,
let $Z(G)$ be the ordinary centre of $G$ and let $Z_{\rm sch}(G)$ be the schematic centre of $G$ which is also the schematic kernel of $\pi$.
The fixed point subgroup $G_{\rm ad}^{\theta_{\rm ad}}$ is a (smooth) reductive subgroup of $G$. Let $\pi^{-1}(G_{\rm ad}^{\theta_{\rm ad}})$
be the (ordinary) inverse image of $G_{\rm ad}^{\theta_{\rm ad}}$ under $\pi$.
Then $H_{\rm a}=Z_{\rm sch}(G)\pi^{-1}(G_{\rm ad}^{\theta_{\rm ad}})$
(see \cite[I.6.2]{Jan} for this notation) is the schematic inverse image of $G_{\rm ad}^{\theta_{\rm ad}}$  under $\pi$.
Let $T_{\rm ad}$ be a maximal torus of $G_{\rm ad}$ which contains a maximal $\theta_{\rm ad}$-split torus, let $P_{\rm ad}$
be a minimal $\theta_{\ad}$-split parabolic subgroup containing $T_{\rm ad}$ and let $B_{\rm ad}$ be a Borel subgroup of $P_{\rm ad}$
containing $T_{\rm ad}$ as in e.g. \cite[Sect.~26.4]{Tim} or \cite[Sect.~1]{DeC-Spr}.
Let $T,P,B$ be the corresponding maximal torus, parabolic and Borel subgroup of $G$.
Now let $H$ be a closed subgroup scheme of $H_{\rm a}$ such that (A1) and (A4) are satisfied. Then $B\cdot x$ is open in $G/H$.

For example, assume $H$ is a closed subgroup scheme of $H_{\rm a}$ with $(G^\theta)^0\subseteq H$ for an involution $\theta$ of
$G$ with $\pi\circ\theta=\theta_{\rm ad}\circ\pi$.\footnote{A $\theta$ with this last property exists when $G$ is semi-simple simply connected, see \cite[9.16]{St}.
}
Then we have
$$\pi^{-1}(G_{\rm ad}^{\theta_{\rm ad}})=\{g\in G\,|\,g\theta(g)^{-1}\in Z(G)\}\,.$$
Since $T$ contains a maximal $\theta$-split torus, it is $\theta$-stable. 
By the arguments from \cite[Lem.~8.1]{Rich} (or \cite[Lemme~2]{V}) $H_{\rm a}(k)$ is generated by $(G^\theta)^0$ and $H_{\rm a}(k)\cap T$.
Adding $Z_{\rm sch}(G)\subseteq T$ to $H_{\rm a}(k)$ gives us $H_{\rm a}$, so (A1) is satisfied in this case.
Now assume in addition $H$ is reduced. We check that (A4) is satisfied.
Since $H=(G^\theta)^0(H\cap T)$ and, by property (*), the character of $H\cap T$ on $T_x(G/H)$ is given by $-2\rho_P$
it is enough to look at the $H^0$-character. 
The $G$-action on the top exterior power of $\Lie(G)$ is trivial and the $H^0$-action on the top exterior power of $\Lie(H)$ is trivial,
so the $H^0$-action on the top exterior power of $\Lie(G)/\Lie(H)$ is trivial. Similarly, the $(G_{\rm ad}^{\theta_{\rm ad}})^0$-action on the top exterior power of
$\Lie(G_{\rm ad})/\Lie(G_{\rm ad}^{\theta_{\rm ad}})$ is trivial. Since $H^0$ maps to $(G_{\rm ad}^{\theta_{\rm ad}})^0$, (A4) follows.


Now we check that the assumptions (A2) and (A3) from Section~\ref{s.symmetric}, which only depend on $H_{\rm a}$, are satisfied.
By \cite[Prop.~3.8, Thm.~3.9]{DeC-Spr}, $G/H_{\rm a}=G_{\rm ad}/G_{\rm ad}^{\theta_{\rm ad}}$ has a wonderful compactification for which
the local structure theorem, with $T$ and $P$ as above, holds, so (A2) is satisfied.
Furthermore, the $T_{\rm ad}=\pi(T)$-orbit map of $x$ is separable by \cite[1.8, 3.5]{DeC-Spr}, so (A3) is satisfied by Lemma~\ref{lem.candiv}.

Finally, we check the assumptions in the results from Section~\ref{s.frobsplit_ratres}. It follows from \cite[Thm.~5.9]{DeC-Spr} that (B1) is satisfied.
Indeed $\tilde\omega_{\bf X}^{-1}$ is ample by \cite[Cor.~17.20, Rem.~17.21]{Tim}, so \cite[Thm.~5.9]{DeC-Spr} and \cite[Thm.~1.2.8]{BriKu} imply that the restriction map
$H^0({\bf X},\tilde\omega_{\bf X}^{-1})\to H^0(G/P^{-},\omega_{G/P^{-}}^{-1})$ is surjective.
Similarly, it is clear from \cite[Prop.~5.7, Thm.~5.10]{DeC-Spr} that (B2) is satisfied
in this situation. Here one has to look at the construction of the good filtration of \cite[Thm.~5.10]{DeC-Spr} to see that it contains a good filtration
for the kernel of the restriction map. 

\section{Parabolic induction}\label{s.induction}
In this final section we show that most of the assumptions that we used are preserved if we apply parabolic induction to $G/H$.
We retain the notation and assumptions (A1)-(A4) of Section~\ref{s.prelim} and~\ref{s.frobsplit_ratres}.
Let $\mc G$ be a connected reductive group and let $\mc Q$ be a parabolic subgroup of $\mc G$ with a surjective homomorphism $\pi:\mc Q\to G$.
We assume that

\smallskip
\begin{enumerate}[({\rm P})]
\item the isogeny $\mc Q/\Ker(\pi)\to G$ induced by $\pi$ is central (cf. \cite[V.2.2]{Bo}).
\end{enumerate}
\smallskip

Let $\mc T$ be a maximal torus of $\mc G$ whose image in $G$ is $T$, let $\mc B$ be a Borel subgroup of $\mc Q$ containing $\mc T$
with $\pi(\mc B)=B$. Let $\mc P\subseteq\mc Q$ be the inverse image of $P$ under $\pi$ and let $\mc Q^-,\mc P^-,\mc B^-$
be the opposites of $\mc Q,\mc P, \mc B$ relative to $\mc T$.\footnote{In \cite{T} $P,\mc P$ and $\mc Q$ were denoted by $Q,\mc Q$ and $\mc P$.}
Let $\mc K$ be the Levi subgroup of $\mc Q$ containing $\mc T$
and denote the natural map $\mc Q^-\to\mc K\to G$ by $\pi^-$. Note that (P) also holds when we replace $\mc Q$ and $\pi$ by
$\mc Q^-$ and $\pi^-$. Now let $\mc H, \mc H_{\rm a}\subseteq \mc Q^-$ be the schematic inverse images under $\pi^-$ of $H$ and $H_{\rm a}$.
Let $\mc K'$ be the subgroup of $\mc K$ generated by $\mc T$ and the simple factors of $\mc K$ that do not lie in the kernel of $\pi$.
Then the isogeny $\mc K'\to G$ is central and $\mc Q^-/\mc H=G/H=\mc K'/\mc H_{\mc K'}$. Note that $\mc Q^-/\mc P^-=G/P^-$, so
$$\mc G\times^{\mc Q^-}G/H=\mc G/\mc H\text{\,\ and \ }\mc G\times^{\mc Q^-}G/P^-=\mc G/\mc P^-.$$
Recall also that for any $\mc Q^-$-variety $X$, the induced variety $\mc G\times^{\mc Q^-}X$ comes with a fibration $pr:\mc G\times^{\mc Q^-}X\to\mc G/\mc Q^-$
for which the fiber over the base point $\mc Q^-$ of $\mc G/\mc Q^-$ is $X$ and with $pr^{-1}\big(R_u(\mc Q)\mc Q^-/\mc Q^-\big)\cong R_u\mc Q\times X$.

By \cite[Lem.~3.1]{T} $\mc G/\mc H$ is spherical, its open $\mc B$-orbit is $\mc B\cdot x$.
Moreover, if $D_i,\,i\in I$ are the $\mc B$-stable prime divisors of $G/H$, then we have
unique $\mc B$-stable prime divisors $\mc D_i=\ov{\mc B D_i},\,i\in I$ of $\mc G/\mc H$ which intersect
$G/H$ in the $D_i$. The other $\mc B$-stable prime divisors of $\mc G/\mc H$ are the pull-backs
$\mc D_\alpha=\ov{\mc Bs_\alpha\cdot x}$ to $\mc G/\mc H$ of the $\mc B$-stable prime divisors
$\ov{\mc Bs_\alpha\mc Q^-/\mc Q^-}$ of $G/\mc Q^-$, $\alpha$ a simple root in $R_u\mc Q$ and $s_\alpha$ the corresponding reflection in the Weyl group of $\mc G$.
We will denote the closures of the $\mc B$-stable prime divisors of $\mc G/\mc H$ in any
$\mc G/\mc H$-embedding by the same letters. Recall that the valuation cone of $\mc G/\mc H$ is the same as that of $G/H$,
see \cite[Sect.~20.6]{Tim} or \cite[Prop.~3.1]{T}, and it is easy to determine the images of the $\mc D_\alpha$ in 
the valuation cone, cf. \cite[Prop.~4.1]{T}. Furthermore, the toroidal embeddings of $\mc G/\mc H$
are obtained by parabolically inducing the toroidal embeddings of $G/H$, see \cite[Prop.~3.2]{T}.

Now we will show that assumptions (A1)-(A4) hold for $\mc G$, $\mc H$, $\mc B$ etc.
If $T_1$ is a closed subgroup scheme of $H_T$ which normalises $H$ such that $T_1$ and $H$ generate $H_{\rm a}$, then the schematic
inverse image of $T_1$ under $\pi:\mc H_\mc T\to H_T$ normalises $\mc H$
and together with $\mc H$ it generates $\mc H_{\rm a}$. Thus (A1) holds for $\mc G,\mc H$ and $\mc H_{\rm a}$.

By (A2) $\mc Q^-/\mc H_{\rm a}=G/H_{\rm a}$ has a wonderful compactification $\bf X$ for which the local structure theorem holds.
So, by \cite[Prop.~3.3]{T}, $\mc G/\mc H_{\rm a}$ has the wonderful compactification $\mbc X=\mc G\times^{\mc Q^-}\bf X$
for which the local structure theorem (relative to $x,\mc T,\mc B$ and with $\mc P$ as above) holds.
Thus (A2) holds for $\mc G/\mc H_{\rm a}$.

\begin{lem}\label{lem.inducedsection}
Let $X$ be an irreducible $\mc K$-variety and let $\mc L$ be a $\mc G$-linearised line bundle on $\mc X=\mc G\times^{\mc Q^-}X$ ($R_u\mc Q^-$ acting trivially on $X$).
Let $\tilde s$ be the $\mc B$-semi-invariant lift of a nonzero $\mc B_\mc K$-semi-invariant rational section $s$ of $L=\mc L|_X$
of weight $\lambda$. Then $\tilde s$ is a global section if and only if $s$ is global and $\lambda$ is dominant relative to $\mc B$.
If $\tilde s$ is global, then the divisor $(\tilde s)$ of $\tilde s$ contains $\mc D_\alpha=\ov{\mc Bs_\alpha\cdot X}$ with coefficient $\la\lambda,\alpha^\vee\ra$
for every simple root $\alpha$ of $R_u\mc Q$.
\end{lem}
\begin{proof}
Since the global sections of $\mc L$ form a rational $\mc G$-module it is clear that for $\tilde s$ to be global $s$ must be global and $\lambda$ dominant.
Now assume the latter. Let $\Delta_{\mc K}(\lambda)$ be the $\mc K$-Weyl module of highest weight $\lambda$. By \cite[II.2.11, 5.21]{Jan}, \cite[Sect.~1]{Don}
we can consider $\Delta_{\mc K}(\lambda)$ as the sub $\mc K$-module of the $\mc G$-Weyl module $\Delta(\lambda)$ generated by the highest weight vector $v_\lambda$.
In fact $\Delta_{\mc K}(\lambda)$ is the sum of the weight spaces of $\Delta(\lambda)$ corresponding to the weights which are congruent to $\lambda$ modulo
the root lattice of $\mc K$. The sum of the other weight spaces is $\mc Q^-$-stable and the quotient by it is isomorphic to
$\Delta_{\mc K}(\lambda)$ with $R_u\mc Q^-$ acting trivially. Using the projection $\Delta(\lambda)\to\Delta_{\mc K}(\lambda)$
and the universal property of the Weyl module $\Delta_{\mc K}(\lambda)$, we obtain a homomorphism of $\mc Q^-$-modules
$f:\Delta(\lambda)\to H^0(X,L)$ with $f(v_\lambda)=s$.
Since $\mc L=G\times^{\mc Q^-}L$ we have $$H^0(\mc X,\mc L)=\ind_{\mc Q^-}^\mc GH^0(X,L)=\Mor_{\mc Q^-}(\mc G,H^0(X,L))\,,$$
$\mc Q^-$ acting by right multiplication on $\mc G$, and Frobenius reciprocity gives us the homomorphism
of $\mc G$-modules $\tilde f:v\mapsto (g\mapsto f(g^{-1}v)):\Delta(\lambda)\to H^0(\mc X,\mc L)$.
Now we have $\tilde s=\tilde f(v_\lambda)=g\mapsto f(g^{-1}v_\lambda)$.
A direct description of $\tilde s$ without the above identification is $\tilde s(g\cdot y)=g\cdot f(g^{-1}v_\lambda)_y$ for all $g\in \mc G$ and $y\in X$.
Let $X_s$ be the nonzero locus of $s$ and let $\alpha$ be a simple root in $R_u\mc Q$. The pull-back prime divisor $\mc D_\alpha=\ov{\mc Bs_\alpha\cdot X_s}$ intersects the open set
$s_\alpha \mc B\cdot X_s=\mc V_\alpha s_\alpha \mc U_\alpha\cdot X_s\cong \mc V_\alpha\times\mc U_\alpha\times X_s$,
where $\mc V_\alpha$ is the product of the root subgroups $\mc U_\beta$ with $\beta>0,\ne\alpha$ and $s_\alpha(\beta)$ a root of $R_u\mc Q$.
Here the isomorphism $\mc V_\alpha\times\mc U_\alpha\times X_s\stackrel{\sim}{\to}\mc V_\alpha s_\alpha \mc U_\alpha\cdot X_s$ is given by
$(u,u',y)\mapsto u n_\alpha u'\cdot y$, where $n_\alpha\in N_\mc G(\mc T)$ is a representative of $s_\alpha$.
If we choose an isomorphism $\theta_\alpha:k_a\to\mc U_\alpha$, $k_a$ the additive group of $k$, and denote the corresponding $k_a$-coordinate by $a$
then $\mc D_\alpha\cap s_\alpha \mc B\cdot X_s$ is defined by $a=0$. Let $u\in\mc V_\alpha$, $y\in X$ and $a\in k$.
Then $\tilde s(un_\alpha\theta(a)\cdot y)=un_\alpha\theta(a)\cdot f(\theta(-a)n_\alpha^{-1}v_\lambda)_y$. Now $w=n_\alpha^{-1}v_\lambda$ is a nonzero weight vector
of weight $s_\alpha(\lambda)=\lambda-m\alpha$, with $m=\la\lambda,\alpha^\vee\ra$, and $\theta(-a)w=\sum_{i=0}^m(-a)^iX_{\alpha,i}w$ in the notation of \cite[II.1.12]{Jan}.
Furthermore, we have $f(X_{\alpha,i}w)=0$ for $i<m$ and $X_{\alpha,m}w=cv_\lambda$ for some $c\in k\sm\{0\}$.
So $\tilde s(un_\alpha\theta(a)\cdot y)=c(-a)^mun_\alpha\theta(a)\cdot s_y$. Since $un_\alpha\theta(a)\cdot y\mapsto un_\alpha\theta(a)\cdot s_y$
is nowhere zero on $s_\alpha \mc B\cdot X_s$ it follows that $(\tilde s)$ contains $\mc D_\alpha$ with coefficient $m$.
\end{proof}

\begin{cornn} 
Assumption (A3) holds for $\mc G/\mc H_{\rm a}$.
\end{cornn}
\begin{proof}
We have $\omega_{\mbc X}=\omega_{pr}\ot pr^*\omega_{\mc G/\mc Q^-}$, where $pr:\mbc X\to\mc G/\mc Q^-$ is the canonical projection
and $\omega_{pr}$ is the relative canonical bundle. Now $pr^*\omega_{\mc G/\mc Q^-}=\mc G\times^{\mc Q^-}({\bf X}\times k_{-2\rho_\mc Q})$ and
$\omega_{pr}=\mc G\times^{\mc Q^-}\omega_{\bf X}$, so $\omega_{\mbc X}^{-1}=\mc G\times^{\mc Q^-}(\omega_{\bf X}^{-1}\ot k_{2\rho_\mc Q})$.
Let $s'$ be the $B$- (or $\mc B_{\mc K}$-) semi-invariant global section of weight $2\rho_P$ of $\omega_{\bf X}^{-1}$ and let $s$ be $s'$ considered as a
($\mc B_{\mc K}$-semi-invariant) global section of $\omega_{\bf X}^{-1}\ot k_{2\rho_\mc Q}$.
Note that we can consider $2\rho_P$ as a character of $\mc T$ and then it is the sum of the roots in $R_u\mc P_\mc K=R_u\mc P_{\mc K'}$, so $\rho_\mc P=\rho_P+\rho_\mc Q$.
Since $pr^{-1}(\mc BQ^-/\mc Q^-)\cong R_u\mc Q\times\bf X$ and since (A3) holds for $G/H_{\rm a}$, it is enough to show that $s$ lifts to a $\mc B$-semi-invariant global section $\tilde s$ of
$\omega_{\mbc X}^{-1}$ of weight $2\rho_{\mc P}$ which vanishes along the prime divisors $\mbc D_\alpha$, $\alpha$ a simple root in $R_u\mc Q$. So the result follows from
Lemma~\ref{lem.inducedsection} with $\lambda=2\rho_\mc P$, noting that $\la\lambda,\alpha^\vee\ra>0$ for $\alpha$ a simple root in $R_u\mc Q\subseteq R_u\mc P$.
\end{proof}

Clearly (A4) is equivalent to

\medskip
{\it\noindent $\omega_{G/H}$ is $G$-equivariantly isomorphic to the pull-back of $\omega_{G/H_{\rm a}}$ along $G/H\to G/H_{\rm a}$.}
\medskip

\noindent But, as in the proof of the above corollary, we have $\omega_{\mc G/\mc H}=(\mc G\times^{\mc Q^-}\omega_{G/H})\ot pr^*\omega_{\mc G/\mc Q^-}$
and $\omega_{\mc G/\mc H_{\rm a}}=(\mc G\times^{\mc Q^-}\omega_{G/H_{\rm a}})\ot pr_{\rm a}^*\omega_{\mc G/\mc Q^-}$, where $pr:\mc G/\mc H\to\mc G/\mc Q^-$ and
$pr_{\rm a}:\mc G/\mc H_{\rm a}\to\mc G/\mc Q^-$ are the canonical projections. From this it follows that (A4) holds for $\mc G$, $\mc H$ and $\mc H_{\rm a}$. 

We now show that assumption (B2) is preserved under parabolic induction.
\begin{prop}\label{prop.induction}
If ${\bf X}$ satisfies (B2), then so does ${\mbc X}$.
\end{prop}
\begin{proof}
By Remark~\ref{rem.cansplit}.2 the assumption is equivalent to the existence of a $B$-canonical splitting of $\bf X$ which compatibly splits the closed orbit $G/P^-$.
Assume the latter holds. Then this splitting is also $B^-$-canonical, by \cite[Prop.~4.1.10]{BriKu}, and therefore also $\mc B^-$-canonical,
when we consider $X$ as a $\mc Q^-$-variety via $\pi:\mc Q^-\to G$. So, by a result of Mathieu, see \cite[Prop.~5.5]{Mat} or \cite[Thm.~4.1.17, Ex.~4.1.E(4)]{BriKu},
$\mc G\times^{\mc B^-}\bf X$ is $\mc B^-$-canonically and therefore also $\mc B$-canonically Frobenius split, compatible with $\mc G\times^{\mc B^-}\mc Q^-/\mc P^-$.
From the fact that the morphism $\mc G\times^{\mc B^-}\bf X\to\mbc X=\mc G\times^{\mc Q^-}\bf X$ is a locally trivial fibration with fiber $\mc Q^-/\mc B^-$ we deduce
easily that the push-forward of the structure sheaf of $\mc G\times^{\mc B^-}\bf X$ is that of $\mbc X$. Now we push the splitting down to $\mbc X=\mc G\times^{\mc Q^-}\bf X$
and obtain that $\mbc X$ is $\mc B$-canonically Frobenius split compatible with the closed $G$-orbit.
\end{proof}
\begin{rem}
It is not clear to me whether (B1) is preserved under parabolic induction,
but recall from Remark~\ref{rem.cansplit}.2 that it is implied by (B2).
\end{rem}

We end with a description of the Picard group of $\mbc X$ in terms of that of $\bf X$ similar to \cite[Prop.~4.2]{T}\footnote{
In the proof of \cite[Prop.~4.2(i)]{T} the first occurrence of $(-\varpi_i,\varpi_i)$ should be replaced by
$\varpi_{i1}+\varpi_{i2}$ and the second occurrence should be omitted.
} 
in the case of induction from a reductive group.
Since we are only interested in $\mbc X$ and $\bf X$ we may assume that $\mc H_{\rm a}=\mc H$ and $H_{\rm a}=H$, and since
$\mc H_{\rm a}$ and $H_{\rm a}$ contain the connected centres of $\mc G$ and $G$ we may assume that $\mc G$ and $G$
are semi-simple and $\mc G$ simply connected. Then ${\rm Pic}(\mbc X)={\rm Pic}_{\mc G}(\mbc X)$ and the restriction of 
$\pi$ to the derived group $D\mc K'$ is the simply connected cover $\tilde G\to G$ of $G$.
It is now easy to see that we obtain a diagram with exact rows as in \cite[p26]{DeC}
$$
\xymatrix{
0\ar[r] & {\rm Pic}(\mc G/\mc Q^-)\ar[r]^{p^*}\ar[d]^{\id}&{\rm Pic}(\mbc X)\ar[r]^{\tilde{h}^*}\ar[d]^{j^*}&{\rm Pic}(\bf X)\ar[r]\ar[d]^{i^*}&0\\
0\ar[r] & {\rm Pic}(\mc G/\mc Q^-)\ar[r]^{(pj)^*}&{\rm Pic}(\mc G/\mc P^-)\ar[r]^{h^*}&{\rm Pic}(G/P^-)\ar[r]&0
}
$$
where the arrows are pull-backs associated to the maps $i:G/P^-\hookrightarrow X$,\\ $\tilde h:X\hookrightarrow\mbc X$,
$h:G/P^-\hookrightarrow\mc G/\mc P^-$, $j:\mc G/\mc P^-\hookrightarrow\mbc X$, and $p:\mbc X\to\mc G/\mc Q^-$.
Here one uses that ${\rm Pic}(R_u\mc Q\times\bf X)={\rm Pic}(\bf X)$ and similar for $G/P^-=\mc Q^-/\mc P^-$.
Furthermore, $j^*$ is injective if $i^*$ is, see \cite[Prop.~4.1]{DeC}. 

As is well-known, ${\rm Pic}(\mc G/\mc P^-)$ and ${\rm Pic}(G/P^-)$
are isomorphic to certain subgroups $\Omega$ and $\Sigma$ of the weight lattices of $\mc G$ and $\tilde G$.
The first is freely generated by the fundamental weights $\omega_\beta$, $\beta$ a simple root in $R_uP$ (or in $R_u\mc P_{\mc K'}=R_u\mc P_{\mc K}$),
and $\omega_\alpha$, $\alpha$ a simple root in $R_u\mc Q$; the second is freely generated by the fundamental weights
$\varpi_\beta$, $\beta$ a simple root in $R_uP$. Furthermore, it is well-known that ${\rm Pic}(\mbc X)$ is freely generated by the
line bundles corresponding to the $\mc B$-stable prime divisors that are not $\mc G$-stable, and similar for ${\rm Pic}(\bf X)$.

Using the fact that ${\rm Pic}({\mbc X}_0)=0$ one can describe $j^*$ (and similarly $i^*$) as follows: a line bundle $\mc L$ on $\mbc X$
is mapped to the weight of the (unique up to scalar multiples) $\mc B$-semi-invariant rational section of $\mc L$ which is nonzero on $\mc G/\mc P^-$. 
In the case of the line bundle corresponding to a $\mc B$-stable prime divisor which is not $\mc G$-stable this is the weight of the canonical section. 
Similar for $\bf X$, where we use the Borel subgroup $\tilde B$ of $\tilde G$ corresponding to $B$.

We describe a right inverse to $h^*$.
The natural restriction map $\Omega\to\Sigma$ sends $\omega_\beta$ to $\varpi_\beta$ and the $\omega_\alpha$ to $0$.
This map has a natural right inverse: for each $\lambda\in\Sigma$ there is a unique $\tilde\lambda\in\Omega$ which restricts to $\lambda$
and with $\la\tilde\lambda,\alpha^\vee\ra=0$, for all simple roots $\alpha$ in $R_u\mc Q$. Under this map $\varpi_\beta$ goes to $\omega_\beta$.

Assume from now on that $i^*$ (and $j^*$) is injective.\footnote{This is the case for symmetric spaces, see \cite[Thm.~4.2(ii)]{DeC-Spr}.}
We describe the right inverse to $\tilde{h}^*$ corresponding to that of $h^*$.
If for the $B$-stable prime divisor ${\bf D}_i$, the line bundle $L({\bf D}_i)$ on $\bf X$ has image $\mu_i$ in $\Sigma$,
then the line bundle $\mc L(\mbc D_i)$ has image $\tilde\mu_i$ in $\Omega$.
Indeed, if $s$ is the canonical section of $\mc L(\mbc D_i)$, then $(s)=\mbc D_i$
and $s|_{\bf X}$ is the canonical section of $L({\bf D}_i)$. So, by Lemma~\ref{lem.inducedsection}, $s$ must have weight $\tilde\mu_i$. 
Thus we can conclude that $L({\bf D}_i)$ is mapped to $\mc L(\mbc D_i)$.
We note that, under this right inverse, the line bundles $L({\bf X}_j)$ corresponding to the boundary divisors ${\bf X}_j$ of $\bf X$ will in general
not be mapped to the line bundles $\mc L({\mbc X}_j)$ of $\mbc X$.
The $L({\bf X}_j)$ have a natural $G$-linearisation and we have $\mc L({\mbc X}_j)=\mc  G\times^{\mc Q^-}L({\bf X}_j)$.
So the images of the $L({\bf X}_j)$ in $\Sigma$ actually lie in the character group of $T$ and the images of the $\mc L({\mbc X}_j)$ in $\Omega$ are obtained by pulling
these back along $\pi:\mc T\to T$.

So in the case $\bf X$ is a symmetric space, one simply obtains ${\rm Pic}(\mbc X)$ by replacing in the formulas at the end of Remark~\ref{rems.candiv}.2 the fundamental weights of $\tilde G$ by the
corresponding fundamental weights of $\mc G$ and then adding to these the fundamental weights $\omega_\alpha$, $\alpha$ a simple root in $R_u\mc Q$.
The images of the $\mc L({\mbc X}_j)$ in $\Omega $ are obtained by pulling the weights $\alpha_j-\theta(\alpha_j)$ back along $\pi:\mc T\to T$.

\medskip

\noindent{\it Acknowledgement}. I would like to thank D.~Timashev and M.~Brion for helpful discussions, and the referee for helpful comments.
This research was funded by the EPSRC grant EP/L013037/1.

\bigskip

{\sc\noindent School of Mathematics,
University of Leeds, Leeds, LS2 9JT , UK.
{\it E-mail address : }{\tt R.H.Tange@leeds.ac.uk}
}

\end{document}